\documentclass[12pt]{amsart}
\usepackage[all,cmtip,2cell]{xy}

\usepackage{amsmath,amsthm,amssymb, paralist, xspace, graphicx, url, amscd, euscript, mathrsfs,epic,eepic,verbatim}
\usepackage{widebar}
\usepackage{enumitem}
\usepackage{colonequals}
\usepackage{caption}

\usepackage[usenames,dvipsnames]{color}

\usepackage{tikz}
\usetikzlibrary{patterns}

\usepackage[colorlinks=true]{hyperref}

\newcommand{\defeq}{\colonequals}

\numberwithin{equation}{subsection}

1

\numberwithin{subsection}{section}

\allowdisplaybreaks[1]

\newtheorem*{namedtheorem}{\theoremname}
\newcommand{\theoremname}{testing}

\newtheorem{theorem}[subsubsection]{Theorem}
\newtheorem{proposition}[subsubsection]{Proposition}
\newtheorem{proposition-definition}[subsubsection]
{Proposition-Definition}
\newtheorem{corollary}[subsubsection]{Corollary}
\newtheorem{lemma}[subsubsection]{Lemma}

\theoremstyle{definition}
\newtheorem{definition}[subsubsection]{Definition}

\newtheorem{example}[subsubsection]{Example}

\newtheorem{remark}[subsubsection]{Remark}

\renewcommand{\thesubsubsection}{\ifnum\value{subsection}=0
	\arabic{section}.\arabic{subsubsection}%
\else
	\arabic{section}.\arabic{subsection}.\arabic{subsubsection}%
\fi}

\makeatletter
\@addtoreset{subsubsection}{section}

\let\c@equation\c@subsubsection

\let\subsection@old\subsection
\def\subsection#1{\ifnum\value{subsubsection}>0 \ifnum\value{subsection}=0
	\setcounter{subsection}{\value{subsubsection}}%
\fi \fi
\subsection@old{#1}}
\makeatother

\mathchardef\ordinarycolon\mathcode`\:
\mathcode`\:=\string"8000
\begingroup \catcode`\:=\active
  \gdef:{\mathrel{\mathop\ordinarycolon}}
\endgroup

\theoremstyle{remark}

\newcommand{\cX}{{\mathcal X}} 
\newcommand{\cY}{{\mathcal Y}} 

\newcommand{\tX}{\widetilde{X}}
\newcommand{\tY}{\widetilde{Y}}
\newcommand{\tcX}{\widetilde{\mathcal{X}}}
\newcommand{\tcY}{\widetilde{\mathcal{Y}}}

\newcommand{\sC}{\mathscr{C}}

\newcommand\cA{\mathcal{A}}
\newcommand\cB{\mathcal{B}}
\newcommand\cC{\mathcal{C}}

\newcommand\cE{\mathcal{E}}

\newcommand\cL{\mathcal{L}}

\newcommand\cO{\mathcal{O}}

\newcommand\cS{\mathcal{S}}

\newcommand\cU{\mathcal{U}}

\def\onorm#1{\wchoice{\mathnormal}{#1}{\woverbar}}
\def\unorm#1{\wchoice{\mathnormal}{#1}{\wunderbar}}
\def\ocal#1{\wchoice{\mathcal}{#1}{\woverbar}}

\def\ufrak#1{\wchoice{\mathfrak}{#1}{\wunderbar}}

\newcommand{\oM}{\onorm{M}}
\newcommand\ocM{\ocal{M}}

\newcommand{\oC}{\onorm{C}}
\newcommand{\ocC}{\ocal{C}}

\newcommand\uC{\unorm{C}}

\newcommand\uf{\unorm{f}}

\newcommand\uS{\unorm{S}}
\newcommand\uT{\unorm{T}}

\newcommand\uX{\unorm{X}}
\newcommand\uY{\unorm{Y}}

\newcommand\uGamma{{\unorm{\Gamma}}}

\renewcommand\AA{\mathbb{A}}

\newcommand\CC{\mathbb{C}}

\newcommand\GG{\mathbb{G}}

\newcommand\NN{\mathbb{N}}

\newcommand\QQ{\mathbb{Q}}
\newcommand\RR{\mathbb{R}}

\newcommand\ZZ{\mathbb{Z}}

\newcommand\fE{\mathfrak{E}}

\newcommand\fL{\mathfrak{L}}
\newcommand\fM{\mathfrak{M}}

\newcommand\fU{\mathfrak{U}}

\newcommand{\ufM}{\ufrak{M}}

\newcommand{\gp}{{\operatorname{gp}}}
\newcommand{\Log}{{\operatorname{Log}}}

\newcommand\bgm{\cB\GG_m}
\newcommand{\Gm}{\GG_m}
\newcommand{\uBGm}{\underline{\smash{\cB\GG}\kern-1pt}\kern1pt_m}

\newcommand\spec{\operatorname{Spec}}
\newcommand\Spec{\operatorname{Spec}}
\newcommand{\fpr}{\mathop{\times}}

\DeclareMathOperator{\Hom}{Hom}

\def\change#1{\begingroup\color{black}#1\endgroup}

\begin{document}

\title{Boundedness of the space of stable logarithmic maps}
\begin{abstract}
We prove that the moduli space of stable logarithmic maps with fixed numerical invariants, from logarithmic curves to a fixed projective target logarithmic scheme with fine and saturated logarithmic structure, is a proper algebraic stack. This was previously known only with further restrictions on the logarithmic structure of the target.
\end{abstract}
\author[Abramovich]{Dan Abramovich}
\author[Chen]{Qile Chen}
\author[Marcus]{Steffen Marcus}
\author[Wise]{Jonathan Wise}
\date{\today}
\thanks{Abramovich is supported in part by NSF grant
   DMS-1162367;
Chen is supported in part by the Simons foundation and NSF grant DMS-1403271; Wise is supported by an NSA Young Investigator's Grant, Award \#H98230-14-1-0107.}

\address[Abramovich]{Department of Mathematics\\
Brown University\\
Box 1917\\
Providence, RI 02912\\
U.S.A.}
\email{abrmovic@math.brown.edu}

\address[Chen]{Department of Mathematics\\
Boston College\\
140 Commonwealth Avenue\\
Chestnut Hill, MA 02467-3806\\
U.S.A.}
\email{q\_chen@math.columbia.edu}

\address[Marcus]{Mathematics and Statistics\\
The College of New Jersey\\
Ewing, NJ 08628\\
U.S.A.}
\email{marcuss@tcnj.edu}

\address[Wise]{University of Colorado, Boulder\\
Boulder, Colorado 80309-0395\\ USA}
\email{jonathan.wise@math.colorado.edu}


\maketitle
\setcounter{tocdepth}{1}
\tableofcontents
\section{Introduction}

We work over \change{$\CC$}.
All logarithmic structures are assumed fine and saturated, and $\Log$ denotes the algebraic stack parameterizing fine and saturated logarithmic structures as in \cite{Olsson_log}.

\subsection{Statement of result}\label{Sec:statement}

We are given a proper, fine and saturated logarithmic scheme $X=(\uX,M)$ with projective underlying scheme $\uX$. In \cite{GS,Chen,AC} a stack $\ocM_\Gamma(X)$ of stable logarithmic maps of numerical type $\Gamma$ is described. The purpose of this paper is to complete a proof of the following theorem:
\begin{theorem}\label{Th:main}
The stack $\ocM_\Gamma(X)$ is a proper Deligne-Mumford stack. The map $\ocM_\Gamma(X) \to \ocM_\uGamma(\uX)$ to the stack of stable maps of $\uX$ with the underlying numerical data $\uGamma = (g,n,\beta)$ is representable.
\end{theorem}

This result was proven in \cite[Theorems 0.1 and 0.2]{GS} under the assumption that the  sheaf of groups $\oM^\gp$ associated to the characteristic monoid $\oM=M/\cO_X^*$ is globally generated. It was also proven in \cite[Theorem 3.15]{AC} under the stronger assumption that  $\oM$ itself is globally generated. Many cases of interest are covered by the results in \cite{GS,AC}, but significantly the case of a toroidal embedding with self intersections is not.  In the cited papers it was hoped that the result would hold in general, which Theorem \ref{Th:main} provides.

The following key properties, which are part of Theorem \ref{Th:main}, were shown in   \cite[Theorem 0.1]{GS} for  Zariski logarithmic structures $X$ and \cite[Theorem 3.15]{AC}  when the characteristic monoid $\oM$ is  globally generated; the general case is proved in \cite{Wise}.

\begin{theorem}[\cite{Wise}]\label{Th:oldproperties}
\begin{enumerate}
\item  $\ocM_\Gamma(X)$ is algebraic and locally of finite type over $\CC$;
\item the map $\ocM_\Gamma(X) \to \ocM_\uGamma(\uX)$ is representable by algebraic spaces.
\end{enumerate}
\end{theorem}

To complete a proof of Theorem \ref{Th:main}, it remains to show that for general $X$,
\begin{enumerate}
 \item the stack $\ocM_{\Gamma}(X)$ is of finite type, see Proposition \ref{Prop:discrete-surjective};
 \item  $\ocM_\Gamma(X)$ is separated and satisfies the weak valuative criterion for properness, see Proposition \ref{Prop:separation}.
\end{enumerate}
The  two  statements above are proven in this paper by reducing to the case where the characteristic monoid  $\oM$ is globally generated. This case was shown in \cite[Corollary 3.11]{AC}, by further reducing to the rank one case treated in \cite{Chen}.

\begin{remark}
In \cite{AC, GS} it is shown that the map $\ocM_\Gamma(X) \to \ocM_\uGamma(\uX)$ is finite under the assumptions made in those papers
\change{. This is shown in general in \cite{qfnew}.  It follows from this statement that}
the stack $\ocM_\Gamma(X)$ is also projective, being finite over $\ocM_\uGamma(\uX)$, which is known to be projective.
\end{remark}

\subsection{Method}
The main problem is boundedness, namely statement (1) listed above. The problem eluded standard approaches of \'etale descent. Instead we use a form of non-flat logarithmic \'etale descent. 

Our strategy is to  use the ``virtual birational invariance" of the moduli spaces,  proven in \cite{AWnew} when $X$ is logarithmically smooth. Specifically, we construct a proper and logarithmically \'etale morphism $Y \to X$ such that the characteristic sheaf $\oM_Y$ is globally generated (Proposition~\ref{Prop:resolution}). We then show that the map of moduli spaces $\ocM(Y) \to \ocM(X)$ is surjective (Proposition~\ref{Prop:surjective}). We further show that for each numerical datum $\Gamma$ on $X$ there is a finite collection of numerical data $\Gamma_i$ on $Y$ such that $\coprod\ocM_{\Gamma_i}(Y) \to \ocM_\Gamma(X)$ is surjective (Proposition~\ref{Prop:discrete-surjective}). Since $\ocM_{\Gamma_i}(Y)$ is proper it follows that $\ocM_\Gamma(X)$ is bounded, as required.

We now proceed to describe the steps in more detail.

\subsection{The Artin fan of $X$}
Olsson \cite{Olsson_log} associates to the logarithmic structure $X$ a canonical morphism $X \to \Log$ to the stack of logarithmic structures.  Under mild assumptions on $X$ there is an initial factorization of this map through a strict, representable, \'etale map $\cX \rightarrow \Log$.  Following \cite{AWnew} we call $\cX$ the \emph{Artin fan of $X$}.
 
The construction of $\cX$ has its origin in unpublished notes on gluing Gromov--Witten invariants by Q.\ Chen and by M.\ Gross.  It is closely related to what is known as the {\em Kato fan} $F(X)$ of $X$ \cite[Sections 9 and 10]{Kato-toric}, and to the associated {\em generalized polyhedral cone complex} $\Sigma(X)$ defined in \cite{Thuillier, ACP}.  A more complete picture of the relationship between these objects, as well as with Berkovich spaces, is given in \cite{Ulirsch}. The simplest cases of Artin fans were used previously
  in \cite{ACFW,ACWnew,CMW}.

We have not attempted to give a definitive treatment of the theory of Artin fans here, as the precise outlines of the theory remain murky to us.  One of the most troublesome issues is the failure of naturality of the morphism from a scheme to its Artin fan (see Example~\ref{ex:non-functoriality}).  Absent the more complete foundations we hope to be able to present in the future, the reader may consult~\cite{AWnew} or~\cite{logsurvey} for further details about Artin fans.

Artin fans are used in the following statement, which is our key reduction step.

\begin{proposition}\label{Prop:resolution}
There exists a representable, projective, birational, and logarithmically \'etale morphism $\cY \to \cX$ such that the sheaf of characteristic monoids $\oM_\cY$ is globally generated. Writing $Y = X \times_\cX \cY$, we have a projective and logarithmically  \'etale morphism $Y \to X$ such that the characteristic sheaf  $\oM_Y$ is globally generated.
\end{proposition}

See Corollary~\ref{cor:target-modify}.

\subsection{Moduli of prestable maps}\label{sec:introprestablemaps}

Following \cite[\change{Section~3}]{AWnew}, we define a moduli stack $\fM(\cX)$ of prestable maps with target $\cX$, a stack $\fM(\cY)$ of prestable maps  with target $\cY$, and a stack $\fM'(\cY \to \cX)$ of prestable maps which are relatively stable for $\cY \to \cX$. All three are shown in \cite[\change{Proposition~3.2 and Proposition 1.6.2}]{AWnew} to be logarithmically \'etale over the stack $\fM$ of prestable curves.   There is a tautological diagram of  stacks 
\begin{equation} \label{eqn:3} \vcenter{\xymatrix{
\ocM(Y) \ar[r]\ar[d] &  		\ocM(X)\ar[d] \\
\fM'(\cY \to \cX)\ar[r]\ar[d] &	\fM(\cX)\\
\fM(\cY)}}
\end{equation}
with strict vertical arrows and cartesian square.  The morphism $\fM'(\cY \to \cX) \to \fM(\cX)$ is birational \cite[\change{Proposition 5.2.1}]{AWnew}. We prove in Corollary \ref{Cor:valuative} that it satisfies the valuative criterion for properness and is surjective. This gives in particular the following.%

\begin{proposition}\label{Prop:surjective} The morphism  $\ocM(Y) \to \ocM(X)$ is proper and surjective.
\end{proposition} 
A \change{direct} argument then shows the following, see Section \ref{Sec:separation}:
\begin{proposition}\label{Prop:separation}
$\ocM(X)$ is separated and satisfies the weak valuative criterion for properness.
\end{proposition}

\subsection{Numerical data}

 If $\ocM(Y)$ were of finite type we would now be done.  As it is anyway a disjoint union of connected components of finite type, it will be sufficient to show that finitely many of those components map to each connected component of $\ocM(X)$.  To do this, we identify numerical data on $\ocM(X)$ that admit finitely many lifts to $\ocM(Y)$\, with each lift corresponding to a component of $\ocM(Y)$.  These numerical data include the genus, the number of marked points, the homology class of the curve, and contact information associated to each marked point.  These are encoded in terms of logarithmic points:

\subsubsection{Moduli of logarithmic points} 
 
The logarithmic numerical data of $X$ are, by definition,  the connected components of the logarithmic evaluation stack $\wedge_\NN X$ parameterizing standard logarithmic points in $X$.  The evaluation stack $\wedge_P X$ for an arbitrary sharp monoid $P$ is constructed and described explicitly in \cite{ACGM}; \change{see} also \cite[Corollary~1.1.3]{Wise}, \cite[Section 3.2]{Gillam}.  Our use of this stack is limited to the case $\wedge_\NN X$, which we denote simply by $\wedge X$.

The formation of $\wedge X$ is covariantly functorial in $X$, so the morphism $Y\to X$ induces a morphism $\wedge Y\to \wedge X$.

\begin{proposition}\label{Prop:component-bounded}
\change{The morphism $\wedge Y \rightarrow \wedge X$ is of finite type, and each connected component of $\wedge X$ is of finite type.  In particular, the preimage of such a connected component has finitely many connected components.}
\end{proposition}

See Section~\ref{sec:contact-orders}.

\subsubsection{Contact orders} \label{sec:intrologpts}

Given a stable log map $f: C \to X$ over $S$, restricting to the $i$-th marking $\Sigma_i \subset C$, we obtain a family of log points $f|_{\Sigma_i}: \Sigma_i \to X$, hence the evaluation morphism $\mathrm{ev}_i: S \to \wedge X$. When $S$ is connected, we label the $i$-th marking by unique connected component of $\wedge X$ containing the image $\mathrm{ev}_i(S)$.  We denote this marking $c_i$ and call it the \emph{logarithmic numerical datum} or \emph{contact order} of the $i$-th marking.

\change{
\begin{remark}
With the notion of the Artin fan $\cX$ of $X$ in hand, one could redefine contact orders as connected component of $\wedge \cX$ instead. This has the advantage of being more combinatorial, while also being invariant under logarithmic modifications, as the proof of Corollary~\ref{cor:log-point-comps} demonstrates.
\end{remark}
}
\subsubsection{Degrees}
To bound $\beta_Y$ we have
\begin{proposition}\label{Prop:deg-constant}
 Let $f: C \to Y$ be a stable logarithmic  map whose stabilization $f': C' \to X$ has
discrete data $\Gamma$. Let $L$ be a relatively ample line bundle for $\cY/\cX$, and
denote by $L_Y$ its pullback to $Y$.  Then $\deg_C f^*L_Y$ is constant on
$\ocM_\Gamma(X)$, and determined combinatorially by~$\Gamma$.
\end{proposition}

 See Proposition \ref{lem:curve-boundedness}. By a standard argument, given $\Gamma$ there are only finitely many possibilities for $\beta_Y$ with image class $\beta$ and fixed $\beta_Y \cdot c_1(L_Y)$:  see Proposition \ref{Prop:bound-class}.
Together with Propositions \ref{Prop:surjective} and \ref{Prop:component-bounded} this implies:
\begin{proposition}\label{Prop:discrete-surjective}
For each numerical datum $\Gamma$ on $X$ there is a finite collection of numerical data $\Gamma_i$ on $Y$ such that $\coprod\ocM_{\Gamma_i}(Y) \to \ocM_\Gamma(X)$ is surjective.
\end{proposition}

\noindent These propositions together provide our main theorem.

\begin{proof}[Proof of Theorem \ref{Th:main}]
By Proposition \ref{Prop:resolution} we have $\oM_Y$ is globally generated, so each $\ocM_{\Gamma_i}(Y)$ is proper by either \cite[Proposition~5.8]{AC} or \cite[Theorem~0.2]{GS}. We rely on the properties enumerated in Section \ref{Sec:statement}. The stack  $\ocM_\Gamma(X)$ is algebraic, locally of finite type over $\CC$,  and separated. Since the image of a proper algebraic stack in a separated algebraic stack is proper, Proposition \ref{Prop:discrete-surjective}  implies that $\ocM_\Gamma(X)$ is proper. The map $\ocM_\Gamma(X) \to\ocM_\uGamma(\uX)$ is representable by Theorem \ref{Th:oldproperties}.
\end{proof}

\section{The stack $\ocM_\Gamma(X)$}\label{Sec:review}
Let $X$ be a logarithmic scheme that is projective over \change{$\CC$} and let $S$ be another logarithmic scheme over \change{$\CC$}. A {\em prestable logarithmic map over $S$ with target $X$} consists of a  logarithmic curve $C \to S$ in the sense of \cite{FKato, Olsson}, along with a logarithmic morphism $C \to X$. It is customary to indicate such a map by $C \to X$, suppressing the remaining data from the notation. A prestable logarithmic map $C \to X$ is {\em stable} if the map $\uC \to \uX$ of underlying schemes is Kontsevich stable.

There are at least three distinct ways to create a groupoid of stable or prestable  logarithmic maps. First, note that given a prestable logarithmic map $C \to X$ over $S$, its pullback  along a logarithmic morphism $S'\to S$ is a logarithmic map $C'\to X$ over $S'$. This defines a groupoid over the category of logarithmic schemes, which we denote temporarily by $\fL(X)$. Second, one can consider only {\em strict} arrows $S'\to S$, namely arrows obtained by pullback along $\uS' \to \uS$. This forms a groupoid over the category of {\em schemes}, by sending a prestable logarithmic map $C \to X$ over $S$ to the scheme $\uS$. We denote this groupoid temporarily by $\fL^{str}(X)$. A key result is the following:

\begin{theorem}[\protect{See \cite[Theorem 2.4]{GS}, \cite[Theorem 2.1.10]{Chen}, \cite[Corollary~1.1.2]{Wise}}]
The groupoid $\fL^{str}(X)$ is an algebraic stack, locally of finite type over~\change{$\CC$}.
\end{theorem}

The stack $\fL^{str}(X)$ has a canonical logarithmic structure, since an object $C \to X$ over $S  = (\uS, M_S)$ defines a logarithmic structure on $\uS$.

The stack $\fL^{str}(X)$ is rather large, because it includes all possible choices of logarithmic structures $S$ on $\uS$, and fails to be proper. A better behaved substack of {\em minimal prestable logarithmic maps} is defined  in \cite{Chen, AC} when the characteristic sheaf $\oM_X$ is globally generated, in \cite{GS} when $X$ is a Zariski logarithmic scheme, and in \cite{Wise} in general. It is denoted $\ufM(X)$. It obtains a canonical logarithmic structure by restriction from $\fL^{str}(X)$, and we denote by $\fM(X)$ this stack with its logarithmic structure; in particular $\fM(X)$ can be viewed as a groupoid over the category of logarithmic schemes. \change{It has} the following key properties:

\begin{theorem}[{\cite[Corollary~1.1.2]{Wise}}]\label{Th:Wise} 
\begin{enumerate} \item The stack $\ufM(X)$ is an open substack of  $\fL^{str}(X)$.  In particular it is algebraic and locally of finite type over \change{$\CC$}. 
\item We have an isomorphism $\fM(X) \simeq \fL(X)$ of groupoids over the category of logarithmic schemes.
\item The morphism $\ufM(X)\to \ufM(\uX)$ is representable by algebraic spaces.
\end{enumerate}
\end{theorem}
This immediately implies Theorem \ref{Th:oldproperties}.

The second statement justifies naming $\fM(X)$ {\em the logarithmic stack of prestable logarithmic maps}. Concretely it says that every prestable logarithmic map $C\to X$ over a logarithmic scheme $S$ is canonically the pullback along a logarithmic morphism $S \to S^{\min}$ of a minimal prestable logarithmic map $C^{\min} \to X^{\min}$ over  $S^{\min}$,  and the underlying map of schemes $\uS \to \uS^{\min}$ is the identity. The first statement then tells us that the groupoid of prestable logarithmic maps with target $X$ is a logarithmic algebraic stack. 

For a prestable map $f:C\to X$ over $S$, we denote by $g$ the arithmetic genus of the fibers of $\uC \to \uS$, by $\beta$  the curve class $\uf_*[\uC]$, and by $c_i$, $i=1,\ldots, n$ the logarithmic numerical data (or {\em contact orders}) of $C\to X$ at the $n$ marked points, introduced in Section \ref{sec:intrologpts}.  These data are locally constant, so $\fM(X)$ breaks into open and closed substacks: $\fM(X) = \coprod_\Gamma \fM_\Gamma(X)$, with $\Gamma = (g,\{c_i\}, \beta)$.

To each prestable logarithmic map $C \to X$ we have an associated map $\uC \to \uX$ of underlying schemes, giving a morphism $\fM(X) \to \fM(\uX)$. This restricts to morphisms $ \fM_\Gamma(X) \to  \fM_\uGamma(\uX)$, where $\uGamma = (g,n,\beta)$. 

Finally we denote by $\ocM(X)$ the open substack of stable logarithmic maps, which again decomposes as $\ocM(X) = \coprod_\Gamma \ocM_\Gamma(X)$.  By definition, the morphism $\fM(X) \to \fM(\uX)$ restricts to $\ocM(X) \to \ocM(\uX)$, and this again decomposes into morphisms $\ocM_\Gamma(X) \to \ocM_\uGamma(\uX)$.

\section{Artin fans}

We extend the construction of \cite{AWnew} to logarithmic schemes which are not logarithmically smooth.

\subsection{The category of Artin fans} \label{sec_cat_of_art}

\change{An \emph{Artin cone} is a logarithmic algebraic stack isomorphic to the quotient of an affine toric variety by its dense torus.  An \emph{Artin fan} is a logarithmic algebraic stack that has a cover by strict, representable, \'etale maps from Artin cones.}
An Artin fan whose tautological morphism to $\Log$ (which is necessarily \'etale) is representable will be said to have \emph{faithful monodromy}.  Logarithmic morphisms between Artin fans are always logarithmically \'etale (see \cite[Lemma~A.7]{AMW}).

It was shown in \cite[Section~2]{AWnew} that a logarithmically smooth scheme $X$ admits an initial factorization of the map $X \rightarrow \Log$ through a representable, \'etale morphism $\cX \rightarrow \Log$.  This stack $\cX$ is called the Artin fan of $X$.  In fact, \change{any} logarithmic scheme with locally connected logarithmic strata admits an Artin fan, as Proposition~\ref{prop:Artin-fan-existence} will show below.

By construction, the Artin fan of a logarithmic scheme always has faithful monodromy.  However, logarithmic modifications (Definition~\ref{def:log-modification}) of Artin fans with faithful monodromy do not necessarily have faithful monodromy.  For this reason, we do not impose a requirement of faithful monodromy in our definition.

If $\sigma$ is a fine, saturated, sharp monoid \change{(i.e., a rational polyhedral cone)} then define $\cA_\sigma$ to be the Artin 
\change{cone $[V/T]$ where $V$ is the affine toric variety associated to $\sigma$ and $T$ is its dense torus.  By \cite[Proposition~5.17]{Olsson_log}, $\cA_\sigma$}
represents the functor
\begin{equation*}
X \mapsto \Hom(\sigma^\vee, \Gamma(X, \oM_X))
\end{equation*}
on logarithmic schemes.  
We write $\cA$ for the Artin cone $\cA_{\NN}$.

Olsson showed that $\Log$ has an \'etale cover by Artin cones \cite[Corollary~5.25, Remark~5.26]{Olsson_log}.\footnote{Note that op.\ cit.\ uses different notation:  $\mathcal{S}_P = \mathcal{A}_\sigma$ where $\sigma = P^\vee$.}  It was shown furthermore in \cite[Corollary~2.2.8]{AWnew} that Artin cones have no nontrivial \change{representable} \'etale covers.  This implies 
that every \change{strict} \'etale map of Artin cones is an open embedding $\cA_\tau \subset \cA_\sigma$ associated to an inclusion of a face $\tau \subset \sigma$
.  As a fiber product of Artin fans is an Artin fan, we conclude that \emph{all Artin fans can be constructed by gluing Artin cones along inclusions of faces}.  More precisely, every Artin fan is a colimit of a diagram in which all morphisms are inclusions of faces.  Note that 
automorphisms \change{are considered inclusions of faces here}, so that Artin fans can be constructed by gluing Artin cones to themselves in nontrivial ways.

Thus Artin fans are essentially combinatorial objects.    In the next section, we give an intuitive guide to the relationship between geometry and combinatorics.  
\change{A precise formulation of the combinatorial nature of Artin fans is given in \cite[Theorem~6.11]{tropmgnnew}.}

\subsubsection{Intuitive picture: the generalized cone complex of an Artin fan} Since Artin fans are 0-dimensional Artin stacks, they are hard to picture. \change{Their relationship with fans (of toric geometry) and cone complexes (of toroidal geometry)
\cite{KKMS}  may be helpful.  
Here we try only to give enough of an idea to motivate the arguments that follow.  For a more detailed discussion, see \cite{logsurvey} or \cite{tropmgnnew} (which are inspired by \cite[Appendix~B]{GS} and \cite{Ulirsch}).}

Given a fine, saturated sharp monoid $\sigma$, it is natural to depict it as a lattice inside the real cone $$\sigma_\RR := \mathrm{conv}(\sigma) \subset \sigma^\gp\otimes_\ZZ\RR.$$ For instance $\NN$ is depicted as the lattice of non-negative integers inside $\RR_{\geq 0}$.

To $\sigma$ we associated  an Artin cone $\cA_\sigma$. We can recover $\sigma$ by the formula $\sigma = \Hom (\NN, \sigma) = \Hom(\cA, \cA_\sigma)$. Since an Artin fan $\cX$ is obtained by gluing Artin cones $\cA_\sigma$ through open embeddings corresponding to face maps $\cA_\tau \to \cA_\sigma$, it is natural to use  the piecewise linear topological space $\Sigma_\cX$, obtained by gluing the cones $\sigma_\RR$, as a concrete avatar of $\cX$.  Note that $\Sigma_{\cX}$ includes not only the cones $\sigma_{\RR}$, but also their lattices $\sigma$, as part of its structure.  When $\cX = [X/T]$, the quotient of the toric variety $X$ of a fan $\Sigma$ by its torus $T$, then $\Sigma_\cX$ is simply the fan $\Sigma$. 

When self-gluing maps are allowed, one does not quite get a complex, but rather a generalized cone complex
\cite{ACP}\change{, or a cone stack~\cite{tropmgnnew}}. The generalized cone complex does not faithfully depict all the subtleties of an Artin fan: see Examples \ref{ex:non-unique-lift} and \ref{ex:non-functoriality}. Nevertheless, the cone complex continues to provide valuable intuition when working with Artin fans.

\subsection{The Artin fan of a logarithmic scheme}
\label{sec:Artin-fans}

\begin{proposition} \label{prop:Artin-fan-existence}
Let $X$ be a logarithmic algebraic stack\footnote{In this paper we will only have use for the case where $X$ is a scheme, but we anticipate it will sometimes be useful to speak of the `Artin fan of an Artin fan'---that is, the universal factorization of the structural morphism $\cX \rightarrow \Log$ through a representable \'etale morphism.  The need for such constructions seems to arise because of the failure of naturality of the construction in this proposition (see Example~\ref{ex:non-functoriality}).}
whose logarithmic strata are locally connected in the smooth topology.  Then there is an initial factorization of the map $X \rightarrow \Log$ through an \'etale morphism $\cX \rightarrow \Log$ that is representable by algebraic spaces.
\end{proposition}
\begin{proof}

Consider the category $\fU'$ consisting of all representable, smooth, and strict morphisms of logarithmic algebraic stacks $U \to X$ with morphisms given by representable, smooth, and strict morphisms of log\change{arithmic} algebraic spaces over $X$. Let $\fU \subset \fU'$ be the subcategory consisting of objects $U \to X$ such that the initial factorization $U \to \cU\to \Log$ of the tautological morphism $U \to \Log$ through an  \'etale morphism $\cU \rightarrow \Log$ that is representable by algebraic spaces  exists.  We aim to show $\fU = \fU'$.

We observe first that $\fU$ is closed under colimits taken in $\fU'$.
Consider a collection of arrows $\{\phi_{ij}:U_i \to U_j\}_{i,j\in \Lambda}$ in $\fU$ with $\Lambda$ a partially ordered set. Assume the colimit $\phi: U = \varinjlim U_i \to X$ exists in $\fU'$. We claim that the colimit $\phi: U \to X$ in fact lies in $\fU$. To see this, let $\cU_i \to \Log$ be the initial factorization of $U_i \to \Log$. The morphisms $\phi_{ij}$ induce morphisms $\psi_{ij}: \cU_{i} \to \cU_{j}$. Let $\cU$ be the colimit of $\{\psi_{ij}\}$ in the category of \'etale sheaves over $\Log$.  We use the symbol $\cU$ also to refer to the espace \'etal\'e of this sheaf, which is an algebraic stack with an \'etale projection to $\Log$ that is representable by algebraic spaces (see~\cite[Theorem~V.1.5]{Milne} for a construction of the espace \'etal\'e.)  Furthermore, the morphism $\cU \to \Log$ serves as the initial representable, \'etale factorization of $U\to \Log$ by the universal property of the colimit.


It is therefore sufficient to show that any geometric point $x$ of $X$ has a neighborhood  in the smooth topology for which the desired factorization exists.  Passing to a smooth-topology neighborhood of $x$ we may assume that $x$ is a member of the closed stratum of $X$, that the closed stratum of $X$ is connected, \change{that the closed stratum is contained in the closure of every connected component of every stratum}, and that 
\begin{equation} \label{eqn:2}
\Gamma(X, \oM_X) \rightarrow \Gamma(x, \oM_X)
\end{equation}
is bijective.  We write $\sigma = \Gamma(X, \oM_X)^\vee$.  We will show that in this situation $\cA_{\sigma}$, with the map $f : X \rightarrow \cA_\sigma$ associated to the bijection $\sigma^\vee \rightarrow \Gamma(X, \oM_X)$, satisfies the required universal property.

Consider a map $g : X \rightarrow \cX$ where $\cX \rightarrow \Log$ is \'etale and representable by algebraic spaces.  We wish to construct a unique map $s : \cA_\sigma \rightarrow \cX$ making the diagram
\begin{equation} \label{eqn:1} \vcenter{\xymatrix{
X \ar[r]^g \ar[d]_f & \cX \ar[d] \\
\cA_\sigma \ar@{-->}[ur]^s \ar[r] & \Log
}} \end{equation}
commute.  Set $\cX' = \cX \fpr_{\Log} \cA_\sigma$.  We observe first that diagram~\eqref{eqn:1} has a unique lift when $X$ is replaced by $x$:
by \cite[\change{Corollary~2.2.8}]{AWnew}, the map 
\begin{equation*}
\Gamma(\cA_\sigma, \cX) \rightarrow \Gamma(x, \cX)
\end{equation*}
is a bijection.   This provides the map $s$ and proves that it is unique; all that is left is to verify that diagram~\eqref{eqn:1} commutes, i.e., that $sf = g$.  

Consider both $sf$ and $g$ as sections of $Z = \cX \fpr_{\Log} X$.  By construction, $sf$ and $g$ agree at $x$.  Because $Z$ is pulled back from the \'etale map $\cX \rightarrow \Log$, it is locally constant on logarithmic strata.  It follows that the locus where $sf$ and $g$ agree is a union of \change{connected components of} strata.  By assumption, the closed stratum of $X$ is connected, so $sf$ and $g$ agree on the closed stratum of $\cX$.  But $Z$ is also \'etale over $X$, and in particular unramified, so the locus in $X$ where $sf$ and $g$ agree is open.  Thus $sf$ and $g$ agree on an open union of strata that contains $x$---that is, they agree on all of $X$.
\end{proof}

\change{We recall the following proposition from \cite[Proposition~2.3.11]{AWnew}:}

\begin{proposition} \label{prop:Artin-fan-minimal}
Let $\cX$ be an Artin fan and let $f : \cA_\sigma \rightarrow \cX$ be a morphism of Artin fans.  Then there is a  factorization of $f$ through a \emph{strict} morphism $\cA_\tau \rightarrow \cX$, which is \change{minimal} with respect to open embeddings.  
\change{The morphism $\cA_\tau \rightarrow \cX$} is unique, up to an $\cX$-isomorphism which is not necessarily unique.  That is, if there is another such factorization through $\cA_{\tau'}$ \change{then there is a isomorphism between $\cA_\tau$ and $\cA_{\tau'}$ over $\cX$, as shown in the commutative diagram:}
\begin{equation} \label{eqn:6} \vcenter{ \xymatrix{
& \cA_{\tau'} \ar[d] \\
\cA_{\tau} \ar[ur]^*!/dr6pt/{\rotatebox{45}{$\scriptstyle\sim$}} \ar[r] & \cX 
}} \end{equation}
\end{proposition}
\change{
\begin{proof}
It is shown in \cite[Proposition~2.3.11]{AWnew} that the triple $(\cA_\tau, p, s)$, where $\cA_\sigma \xrightarrow{s} \cA_\tau \xrightarrow{p} \cX$ is a factorization of $\cA_\sigma \rightarrow \cX$ through a strict morphism $p$, is unique up to \emph{unique} isomorphism.  It follows, therefore, that the pair $(\cA_\tau, p)$ is unique up to isomorphism (and the isomorphisms are in bijection with the choices of $s : \cA_\sigma \rightarrow \cA_\tau$ lifting $\cA_\sigma \rightarrow \cX$ along $p$).
\end{proof}
}

\begin{remark}
In the theory of Kato fans \cite[Sections 9 and 10]{Kato-toric}, \change{namely when self gluing is not allowed,} the lift in diagram~\eqref{eqn:6} is unique.  
\change{In view of \cite[Proposition~2.3.11]{AWnew}, this is because when $\cX$ is a Kato fan, once $\cA_\tau \rightarrow \cX$ has been specified, there is a unique lift of $\cA_\sigma \rightarrow \cX$ to $\cA_\tau$.}
\end{remark}

\begin{example} \label{ex:non-unique-lift}
We give an example involving a cone glued to itself in which the map $\cA_\tau \rightarrow \cX$ in the Proposition is not unique.  Let $\cA^{[2]}$ be the image in $\Log$ of the \'etale map $\cA^2 \rightarrow \Log$.  If we regard $\cA^2$ as the moduli space of logarithmic structures with a global chart by $\mathbb{N}^2$ then we may interpret $\cA^{[2]}$ as the moduli space of logarithmic structures that admit a chart by $\mathbb{N}^2$ \emph{\'etale-locally}.  This arises as the Artin fan of the punctured Whitney umbrella (see Example~\ref{ex:non-functoriality}).

The diagonal gives a \emph{non-strict} map $\cA \rightarrow \cA^2$, which is the minimal factorization of the composition $\cA \rightarrow \cA^2 \rightarrow \cA^{[2]}$.  \change{We note that there are \emph{two} ways to complete the following diagram:
\begin{equation*} \xymatrix{
& \cA^2 \ar[d] \\
\cA^2 \ar@{-->}[ur] \ar[r] & \cA^{[2]}
} \end{equation*}
One may take the diagonal arrow to be either of the two automorphisms of $\cA^2$.}
Note that the generalized cone complex of $\cA^{[2]}$ is simply the quotient cone $(\RR_{\geq 0})^2 / (\ZZ/2\ZZ)$, and the diagram corresponds to the involution of the cone $(\RR_{\geq 0})^2$.

\change{However, these automorphisms induce distinct commutative squares}
\begin{equation*} \xymatrix{
\cA \ar[r]^{\change{\Delta}} \ar[d]_{\change{\Delta}} & \cA^2 \ar[d] \\
\cA^2 \ar[ur] \ar[r] & \cA^{[2]}
} \end{equation*}
\change{because to specify such a square involves the choice of an automorphism of the composition $\cA \xrightarrow{\Delta} \cA^2 \rightarrow \cA^{[2]}$.}
\end{example}

\subsection{A substitute for functoriality of Artin fans}
\label{sec:functoriality}

While Artin fans are functorial with respect to strict morphism of logarithmic schemes, they are not functorial with respect to general logarithmic morphisms.  In this section we adapt the construction of Section~\ref{sec:Artin-fans} to achieve a weak form of functoriality that will be suitable for our application in Proposition~\ref{prop:alteration-lift}.

\begin{example} \label{ex:non-functoriality}
We show that our construction of the morphism from a logarithmic scheme to its Artin fan cannot be natural.  This example is recounted in greater detail in \cite{logsurvey}.

We work over an algebraically closed field of characteristic other than $2$.  The \emph{punctured Whitney umbrella} $X$ is the quotient of $Y = \Gm \times \AA^2$ by the involution exchanging $(t,x,y)$ and $(-t,y,x)$.  We equip $Y$ with the logarithmic structure pulled back from $\AA^2$, which descends to a logarithmic structure on $X$.  The Artin fan of $Y$ is $\cY = \cA^2$ and the Artin fan of $X$ is $\cX = \cA^{[2]}$ (see Example~\ref{ex:non-unique-lift} for the notation).

Let $\tY$ be the blowup of $Y$ along $\Gm \times \{ 0 \}$ and let $\tX$ be the corresponding blowup of $X$.  The Artin fan of $\tY$ is the blowup $\tcY$ of $\cA^2$ at its origin, or, more explicitly, the quotient of the blowup of $\AA^2$ by its dense torus.  The Artin fan of $\tX$ is the quotient $\tcX$ of $\tY$ by the involution exchanging the coordinates as a \emph{representable}, \'etale algebraic stack over $\Log$.  Even though the involution stabilizes the exceptional divisor of $\tY$, the corresponding divisor of the quotient has no additional stabilizer because the map from the Artin fan to $\Log$ must be representable.  This is the reason functoriality fails.

One can now show that there is no dashed arrow completing the diagram below and making it commute:
\begin{equation*} \xymatrix{
\tX \ar[r] \ar[d] & \tcX \ar@{-->}[d] \\
X \ar[r] & \cX
} \end{equation*}
Indeed, there is a loop in the exceptional divisor of $\tX$ that projects to a loop in $X$ around which the characteristic monoid of the logarithmic structure of $X$ has monodromy.  Its image in $\cX$ is therefore gives a nontrivial element of the stabilizer of the closed point of $\cX$.  However, the logarithmic structure of $\tX$ does not have monodromy around this loop since it has rank~$1$ and the characteristic monoid of a rank~$1$ logarithmic structure cannot have monodromy.  Therefore this loop projects to the trivial automorphism in the stabilizer group of $\tcX$.
\end{example}

Let $X$ be a scheme equipped with a \emph{morphism} of logarithmic structures $M'_X \rightarrow M_X$.  Let $\Log_{\Delta^1}$ be the universal example of an algebraic stack with these data \cite[Theorem~2.4]{Olsson-log-cc}, so that there is a tautological map $X \rightarrow \Log_{\Delta^1}$.  We show that there is an initial factorization of this map through a representable \'etale map $\cX \rightarrow \Log_{\Delta^1}$:

\begin{proposition} \label{prop:Artin-fan-existence-2}
Let $X$ be an algebraic stack equipped with a morphism of logarithmic structures such that the logarithmic strata\footnote{By \change{the logarithmic strata} we mean the strata in the coarsest stratification over which the characteristic monoids of both logarithmic structures are locally constant.} are locally connected in the smooth topology.  The corresponding map $X \rightarrow \Log_{\Delta^1}$ admits an initial factorization through a representable \'etale map $\cX \rightarrow \Log_{\Delta^1}$.
\end{proposition}
\begin{proof}
The structure of the proof is essentially the same as that of Proposition~\ref{prop:Artin-fan-existence}, so we omit some details.

We begin by noting that the collection of all smooth $Y \rightarrow X$ such that $Y \rightarrow \Log_{\Delta^1}$ has an initial factorization through a representable, \'etale map $\cY \rightarrow \Log_{\Delta^1}$ is closed under colimits.  As the universal property characterizing this factorization respects colimits, it will be sufficient to work smooth-locally in $X$.  We may therefore assume that there is a closed geometric point $x$ of $X$ for which the maps
\begin{gather*}
\Gamma(X, \oM_X) \rightarrow \Gamma(x, \oM_X) \\
\Gamma(X, \oM'_X) \rightarrow \Gamma(x, \oM'_X)
\end{gather*}
are bijections.  Set $\sigma = \Gamma(X, \oM_X)^\vee$ and $\tau = \Gamma(X, \oM'_X)^\vee$.  The map $\sigma \rightarrow \tau$ induces a map $\varphi : \cA_\sigma \rightarrow \cA_\tau$ and moreover gives a map $\cA_\sigma \rightarrow \Log_{\Delta^1}$.  In order to emphasize the map to $\Log_{\Delta^1}$, we write $\cA_{\sigma \rightarrow \tau} \simeq \cA_\sigma$ here.  

\begin{lemma}
The map $\cA_{\sigma \rightarrow \tau} \rightarrow \Log_{\Delta^1}$ is \'etale \change{and representable,} and the collection of all such maps is an \'etale cover of $\Log_{\Delta^1}$.
\end{lemma}
\begin{proof}
\change{To see that $\cA_{\sigma\rightarrow\tau} \rightarrow \Log_{\Delta^1}$ representable, we interpret a map $S \rightarrow \Log_{\Delta^1}$ as a morphism of logarithmic structures $M'_S \rightarrow M_S$ on $S$.  The lifts to a map $S \rightarrow \cA_{\sigma \rightarrow \tau}$ correspond to commutative diagrams~\eqref{eqn:4} 
\begin{equation} \label{eqn:4} \vcenter{\xymatrix{
\tau^\vee \ar[r] \ar[d] & \oM'_{S} \ar[d] \\
\sigma^\vee \ar[r]  & \oM_{S}
}} \end{equation}
that lift locally to charts.  These are clearly indexed by a set (with no nontrivial automorphisms).
}

The morphism is \'etale if and only if it is locally of finite presentation and formally \'etale.  It is locally of finite presentation because both source and target are locally of finite presentation over \change{$\CC$}.  To verify the infinitesimal lifting property, consider a diagram
\begin{equation} \label{eqn:5} \vcenter{\xymatrix{
S \ar[r] \ar[d] & \cA_{\sigma \rightarrow \tau} \ar[d] \\
S' \ar[r] \ar@{-->}[ur] & \Log_{\Delta^1} 
}} \end{equation}
in which $S'$ is an infinitesimal extension of $S$.  The map $S' \rightarrow \Log_{\Delta^1}$ gives a morphism of logarithmic structures $M'_{S'} \rightarrow M_{S'}$ on $S'$.  The commutativity of the square induces a commutative square~\eqref{eqn:4}
where the vertical arrow on the right is the restriction of the map of characteristic monoids $\oM'_{S'} \rightarrow \oM_{S'}$ to $S$ and the horizontal arrows are charts.  But \change{$S$ and $S'$ have identitical \'etale sites} 
and under this identification $\oM_S = \oM_{S'}$ and $\oM'_S = \oM'_{S'}$.  With these substitutions,~\eqref{eqn:4} \change{can be lifted, \'etale-locally, to a chart because the maps $M_S \rightarrow \oM_S$ and $M'_S \rightarrow \oM'_S$ are surjections of \'etale sheaves.  This}
gives the diagonal arrow lifting~\eqref{eqn:5} and shows it is unique.

The assertion that the $\cA_{\sigma \rightarrow \tau}$ cover $\Log_{\Delta^1}$ translates into the following familiar facts
\begin{enumerate}[label=\arabic{*})]
\item the characteristic monoid of a fine, saturated logarithmic structure possesses a chart locally, and
\item a morphism of fine, saturated logarithmic structures with charts by $\sigma^\vee$ and $\tau^\vee$ may be induced locally from a morphism $\sigma \rightarrow \tau$.
\end{enumerate}
\end{proof}

Returning to the proof of Proposition~\ref{prop:Artin-fan-existence-2}, our reduction guarantees we have a morphism $X \rightarrow \cA_{\sigma \rightarrow \tau}$  over $\Log_{\Delta^1}$.  To see that $\cA_{\sigma \rightarrow \tau}$ 
\change{is the initial factorization of $\cX \rightarrow \Log_{\Delta^1}$ through a strict, \'etale, representable morphism}
we repeat the argument of Proposition~\ref{prop:Artin-fan-existence}.  We consider a commutative diagram
\begin{equation*} \xymatrix{
X \ar[r] \ar[d] & \cX \ar[d] \\
\cA_{\sigma \rightarrow \tau} \ar@{-->}[ur] \ar[r] & \Log_{\Delta^1}
} \end{equation*}
in which $\cX$ is \change{strict, \'etale, and representable} over $\Log_{\Delta^1}$.  Replacing $\cX$ with $\cX \fpr_{\Log_{\Delta^1}} \cA_{\sigma \rightarrow \tau}$, we immediately reduce to the case where there is a map $\cX \rightarrow \cA_{\sigma \rightarrow \tau}$ that is compatible with the rest of the diagram, and the problem is to show there is exactly one section of this map making the rest of the diagram commute.  By assumption, a unique section exists at the geometric point $x$ of $\cA_{\sigma \rightarrow \tau}$.  By \cite[\change{Corollary~2.2.8}]{AWnew}
such a section extends uniquely to a section over $\cA_{\sigma \rightarrow \tau} \simeq \cA_\sigma$.  This completes the proof of Proposition~\ref{prop:Artin-fan-existence-2}.
\end{proof}

\begin{corollary} \label{cor:Artin-fan-functoriality}
Let $\change{Y \rightarrow X}$ be a morphism of logarithmic schemes.  Suppose that $\cX$ is the Artin fan of $X$ and $\cY$ is the Artin fan of $Y$ relative to $\Log_{\Delta^1}$.  Then there is a canonical morphism $\cY \rightarrow \cX$ making the diagram below commute:
\begin{equation*} \xymatrix{
Y \ar[r] \ar[d] & \cY \ar[d] \\
X \ar[r] & \cX
} \end{equation*}
\end{corollary}

\section{Subdivisions}

The goal of this section is to show that essentially any logarithmic scheme has a projective logarithmic modification with globally generated characteristic monoid, see Theorem \ref{thm:subdiv}.  We begin by defining our terms.

\begin{definition}\label{def:log-modification}
\begin{enumerate}
\item A \emph{logarithmic alteration} of logarithmic Artin stacks is a proper, surjective, logarithmically \'etale morphism.

\item A \emph{logarithmic modifiction} of a \emph{logarithmically smooth} Artin stack is a proper, birational, logarithmically \'etale morphism.

\item More generally, a logarithmic alteration $Y \to X$ of logarithmic Artin stacks  is said to be a \emph{logarithmic modifiction}  if there is a logarithmic modification $\cY \to \cX$ of logarithmically smooth Artin stacks and a morphism $X \to \cX$ such that $Y = \cY\times_\cX X$, the product taken in the category of fs logarithmic stacks.
\end{enumerate}
\end{definition}

As the pullback of a logarithmic modification in the sense of \ref{def:log-modification}~(2) to a logarithmically smooth base is a  logarithmic modification  in the sense of \ref{def:log-modification}~(2), Definitions  \ref{def:log-modification}~(2) and  \ref{def:log-modification}~(3) are consistent.

\begin{remark}
\change{F.\ Kato has given a different definition of logarithmic modifications~\cite[Definition~3.14]{Kato-log-mod}.  It is immediate that representable logarithmic modifications in our sense are logarithmic modifications in Kato's sense, but we do not know if the converse holds.  It follows from \cite[Corollary~2.6.6]{AWnew} and Proposition~\ref{prop:subdiv-proper}, below, that the definitions coincide for representable logarithmic modifications of logarithmically smooth schemes.}
\end{remark}

Examples of logarithmic modifications appear in Sections \ref{sec:star} and~\ref{sec:barycentric} below. 


The pullback of a logarithmic alteration is a logarithmic alteration, and the pullback of a  logarithmic modification is a logarithmic modification.
A \emph{representable}  logarithmic modification of logarithmically smooth Artin stacks is a modification in the usual sense, but in general logarithmic modifications need not be representable: they include certain root stack constructions.  

\subsection{Subdivisions of Artin fans}

In \cite{Kato-toric} Kato described certain logarithmic modifications in terms of {\em subdivisions} of Kato fans, in analogy to subdivisions of fans of toric varieties, and we borrow the same analogy and 
\change{define subdivisions of Artin fans.}

\change{By definition} an Artin fan $\cX$ is covered by  strict \change{\'etale} maps $\cA_\sigma \rightarrow \cX$. 
An inclusion of faces $\sigma \subset \tau$ induces a strict open embedding $\cA_\sigma \subset \cA_\tau$, and the assignment $\sigma \mapsto \cA_\sigma$ respects intersections of faces. Therefore, given a fan in the sense of \cite[Section~1.4]{F} or \cite[Definition~3.1.2]{tv}, we may define an Artin fan $\cA_\Sigma$   by gluing together the $\cA_\sigma$ for $\sigma \in \Sigma$ according to the way they intersect inside of $\Sigma$. This permits us to give
\begin{definition}
A {\em subdivision} of an Artin fan $\cX$ is a morphism of Artin fans $\cY \rightarrow \cX$ whose base change via any map $\cA_\sigma \rightarrow \cX$ is isomorphic to $\cA_\Sigma$ for some subdivision $\Sigma$ of $\sigma$.
\end{definition}

Since the morphisms $\varphi : \cA_\sigma \rightarrow \cX$ cover $\cX$, we may construct a map $\cY \rightarrow \cX$ by constructing compatible maps $\cY_\varphi \rightarrow \cA_\sigma$.  The meaning of compatibility here is that the $\cY_\varphi$ should be stable under pullback via face maps $\cA_\tau \rightarrow \cA_\sigma$.  We use this idea to construct several refinements of~$\cX$.

 A subdivision of Artin fans corresponds to a subdivision of generalized cone complexes, so while   Artin fans (or Kato fans) and their subdivisions can be hard to  visualize, when one passes to generalized cone complexes one can actually draw a picture.

\begin{proposition} \label{prop:subdiv-proper}
A representable, birational morphism of Artin fans is proper if and only if it is a subdivision.
\end{proposition}
\begin{proof}
Let $\cY \to \cX$ be a  proper, birational and representable morphism of connected Artin fans. Since the statement is local on $\cX$, replacing $\cX$ by an \'etale local chart, we may assume that $\cX = \cA_{\sigma}$.  We then have a strict global quotient morphism $\AA_{\sigma} \to \cA_{\sigma}$, where $\AA_{\sigma}$ is the affine toric variety associated to $\sigma$ with the maximal torus $T$. We  obtain a $T$-equivariant morphism
\[
h: Y_{\sigma} \defeq \cY \times_{\cX}\AA_{\sigma} \to \AA_{\sigma}.
\]
Since $\cY \rightarrow \cX$ is birational, it is an isomorphism over the generic point, which pulls back to the dense torus in $\AA_\sigma$.
This implies that $Y_{\sigma}$ is toric as well. By the $T$-equivariance and properness of $h$, we deduce that $Y_{\sigma}$ is the toric variety obtained from a subdivision of $\sigma$. This finishes the proof.
\end{proof}

\subsection{Star subdivision}
\label{sec:star}

Let $\sigma$ be a fine, saturated, sharp monoid and $x \in \sigma$ an element.  For each face $\tau$ of $\sigma$ not containing $x$, let $\tau'$ be the saturated submonoid of $\sigma$ generated by $\tau$ and $x$.  The $\tau'$ and all of their faces constitute a fan, called the star subdivision of $\sigma$, and denoted $x\cdot \sigma$.

This construction is functorial with respect to inclusion of faces containing $x$.  That is, if $\sigma \subset \tau$ is the inclusion of a face containing $x$, then $x\cdot \sigma$  is canonically a subfan of $x\cdot \tau$.

We will generalize star subdivision to Artin fans by attempting to glue together star subdivisions of Artin cones.  If $\cA_\sigma \rightarrow \cX$ is an \'etale chart and $x \in \sigma$ is an element at which we would like to subdivide, we must require that the resulting subdivision be \change{compatible with different choices of chart $\cA_\sigma \to\cX$.}  This translates into the condition that $x$ be stable under monodromy.  In order to state things in a way that is intrinsic to Artin fans, we replace vectors \change{$x \in \sigma$} with \change{an} equivalent notion in the language of Artin fans.  

As introduced in the opening paragraphs of Section~\ref{sec_cat_of_art} we write $\cA=\cA_\NN$. The following definition is adapted from~\cite[Section~5.3]{Wlodarczyk}. 
\begin{definition}
Let $\cX$ be an Artin fan.  We will call a morphism of Artin fans $\change{x :} \cA \rightarrow \cX$ a \emph{vector} of $\cX$.  We call a vector $x$ of $\cX$ \emph{stable} if, \change{whenever $\cA_\sigma \rightarrow \cX$ is strict, there is at most one vector of $\cA_\sigma$ whose image is isomorphic to $x$.}
\end{definition}

Thus a vector of $\cX$ is simply a lattice point of the generalized cone complex $\Sigma_\cX$. The following are two examples \change{of}  vectors which are not stable arise.
\begin{example}\label{Ex:nodal} Consider the Artin fan $\cX$ associated to a surface with logarithmic structure given by an irreducible nodal curve with one node. The generalized cone complex is obtained by taking $(\RR_{\geq 0})^2$ and gluing together its two rays. The images of the vectors $(1,0)$ and $(0,1)$ are the same, so the factorization $\NN \to \NN^2 \to \Sigma_\cX$ is not unique. In fact a vector is stable if and only if it is not an image of $(a,0)$ (or $(0,a)$) for positive $a$.
\end{example}
\begin{example}\label{Ex:whitney}  Consider the Artin fan $\cA^{[2]}$ of the punctured whitney umbrella, see Example \ref{ex:non-functoriality}. The generalized cone complex is $(\RR_{\geq 0})^2 / (\ZZ/2\ZZ)$, so a vector is stable if and only if it is the image of a diagonal vector $(a,a) \in \NN^2$. 
\end{example}

Assuming $x$ is stable, we construct the star subdivision $\cX'$ as follows:  For any map $\varphi : \cA_\sigma \rightarrow \cX$, let 
\begin{equation*}
\cX'_\varphi = \begin{cases} \cA_\sigma & \text{$x$ does not lift to $\cA_\sigma$} \\ \cA_{x\cdot \sigma} & \text{$x$ lifts to $\cA_\sigma$} \end{cases}
\end{equation*}
where $\cA_{x\cdot\sigma}$ denotes the star subdivision of $\cA_\sigma$ with respect to the unique lift of $x$ to $\cA_\sigma$.  Since $x$ is stable, this construction is compatible with strict $\cX$-morphisms $\cA_\tau \rightarrow \cA_\sigma$, hence glues to give a global construction.

\begin{proposition}\label{prop:star-subdivision}
Star subdivision is projective.
\end{proposition}
\begin{proof}
Let $\phi_x: \cX' \to \cX$ be a star subdivision given by a stable vector $x: \cA \to \cX$. Note that $\phi_x$ is representable and birational. It suffices to produce a $\phi_x$-ample line bundle over $\cX'$. Let $E \subset \cX'$ be the exceptional divisor. Since $x$ is stable, such $E$ is a well-defined prime divisor over $\cX'$. We first notice that $E$ is $\QQ$-cartier. This could be checked locally via the toric geometry over each chart, see \cite[11.1.6(b)]{tv}. Let $L$ be the line bundle associated to $- k\cdot E$ for some sufficiently divisible positive integer $k$. By \cite[4.6.4]{EGAII} to see that $L$ is $\phi_x$-ample, it suffices to check the statement locally over $\cX$. By taking base change to a covering of $\cX$, we may assume that $\cX = \cA_{\sigma}$. Note that $\cA_{\sigma}$ is given by the global quotient of the affine toric variety $\AA_{\sigma}$ by its maximal torus. The ampleness follows from the fact that star subdivisions induce equivariant projective modifications of toric varieties in which the Cartier divisor $-kE$ is ample. 
\end{proof}

\subsection{Barycentric subdivision}
\label{sec:barycentric}

For a fine, saturated, sharp monoid $\sigma$, let $B(\sigma)$ be the barycentric subdivision of $\sigma$ (see, e.g., \cite[Example II.2.1]{KKMS}, \cite[Exercise~11.1.10]{tv}).  The fan $B(\sigma)$ is automatically simplicial. We obtain a map $\cA_{B(\sigma)} \rightarrow \cA_\sigma$ that is stable under base change via face maps, by definition.  By descent we obtain a map $B(\cX) \rightarrow \cX$ that we call the barycentric subdivision of $\cX$.

\begin{proposition} \label{prop:subdiv-projective}
The barycentric subdivision of a quasi-compact Artin fan is a projective morphism.
\end{proposition}
\begin{proof}
We describe the barycentric subdivision  as a sequence of star subdivisions as follows. The {\em barycenter} $b_\sigma$ of a cone $\sigma$ is the sum of generators of its 1-dimensional faces. To obtain the barycentric subdivision one first star subdivides, in arbitrary order, at the barycenters of cones of maximal dimension $n$, then at the barycenters of the original  cones of dimension $n-1$, etc. 

We claim that these barycenters are stable. First, \change{if $b$ is the image of}
the barycenter $b_\tau$ of an $n$-dimensional cone $\tau$ 
\change{in $\cX$ then $b$}
is stable: 
\change{indeed, if $\cA_\sigma \rightarrow \cX$ were another strict \'etale map to which $b$ lifts then $\sigma$ is be isomorphic to $\tau$ and the barycenter is stable under isomorphism, by construction.}
The barycenter of an $n-1$ dimensional cone $\tau$ is stable in the resulting subdivision, since a factorization $\cA \to \cA_\sigma \to \cX$ either has  $\sigma = \tau$, in which case the factorization is unique as above, or $\sigma = \langle\tau, b_{\tilde\sigma}\rangle$, the cone generated by $\tau$ and the barycenter of an $n$-dimensional  cone  $\tilde\sigma$, which is stable. So two such factorizations can differ only by an automorphism of 
\change{$\langle\tau, b_{\tilde\sigma}\rangle$, but such an automorphism must fix $b_{\tilde \sigma}$ because it is stable.  Therefore two factorizations differ by an automorphism of $\tau$, and once} again they coincide since  $b_\tau$ is invariant
\change{under automorphisms of $\tau$}. Inductively, after star subdividing at all barycenters of $m$-cones for $m>k$, any factorization $\cA \to \cA_\sigma \to \cX$ of $b_\tau$ has $\sigma =  \langle\tau, b_1,\ldots,b_l\rangle$ with  $b_1,\ldots,b_l$ stable. Once again two such factorization can differ only by an automorphism of $\tau$ and again they coincide since  $b_\tau$ is invariant, hence $b_\tau$ is stable.

   Since the barycentric subdivision may be achieved as the composite of  star subdivisions, it is projective. 
\end{proof}

\subsection{Resolution}

The following lemma is essentially a restatement of \cite[Lemma~2.4.6~(1)]{AT2}:

\begin{lemma} \label{lem:barycentric-stable}
Let $\cX$ be an Artin fan and $B\cX$ its barycentric subdivision.  Every vector $x : \cA \rightarrow B \cX$ is stable.
\end{lemma}
\begin{proof}
Suppose we have a strict map $\cA_\sigma \rightarrow B \cX$ and two maps $x,y : \cA \rightarrow \cA_\sigma$ that have isomorphic images in $B \cX$.  Let $\tau$ be the minimal face of $\sigma$ containing $x$. Then symmetrically we also have a minimal face $\tau'$ of $\sigma$ containing $y$ with $\tau \cong \tau'$. We thus have a pair of maps $\cA_\tau \rightrightarrows \cA_\sigma$ that compose to the same map $\cA \to B \cX$ according to the two faces $\tau$ and $\tau'$ of $\sigma$. But recall that $\sigma$ is simplicial, say of dimension $d$, and corresponds to a flag of $d$ faces of a monoid $\omega$, for some strict $\cA_\omega \rightarrow \cX$.  A face $\tau \subset \sigma$ is characterized uniquely by the dimensions of the faces in the corresponding subflag of $\omega$.  A fortiori, the two inclusions $\tau \subset \sigma$ must coincide.
\end{proof}

The following theorem follows the argument of  \cite[Proposition 2.4.1]{AT2}. Step 3 is based on \cite[Lemma 8.7]{AMR}.

\captionsetup{width=.45\textwidth}
\begin{figure}

\begin{minipage}[b]{.5\textwidth}

\begin{center}

\begin{tikzpicture}

\begin{scope}[shift={(0,-4)}]

\draw (-3,1) -- (0,0);
\fill (-3,1) circle (0.030cm);

\def\loopleftside{(0,0) .. controls (0,0.25) and (-.3,0.3) .. (-.3,1) .. controls (-.3,1.5) and (-.3,2) .. (0,2)}
\def\looprightside{(0,0) .. controls (0,0.25) and (.3,0.3) .. (.3,1) .. controls (.3,1.5) and (.3,2) .. (0,2)}

\fill [pattern=north east lines,pattern color=black!10!white] (-3,1) -- \looprightside -- (-3,1);
\fill [fill=white] (-3,1) -- \loopleftside -- (-3,1);
\fill [pattern=north west lines,pattern color=black!30!white] (-3,1) -- \loopleftside -- (-3,1);

\draw [dotted] \loopleftside;
\draw [dotted] \looprightside;

\end{scope}

\draw[thick,->] (-1.5,-0.5) -- (-1.5, -1.5);

\def\loopleftside{(0,0) .. controls (0,0.25) and (-.3,0.3) .. (-.3,1) .. controls (-.3,1.7) and (0,1.75) .. (0,2)}
\def\looprightside{(0,0) .. controls (0,0.25) and (.3,0.3) .. (.3,1) .. controls (.3,1.7) and (0,1.75) .. (0,2)}

\fill [pattern=north east lines,pattern color=black!10!white] (-3,1) -- \looprightside -- (-3,1);
\fill [fill=white] (-3,1) -- \loopleftside -- (-3,1);
\fill [pattern=north west lines,pattern color=black!30!white] (-3,1) -- \loopleftside -- (-3,1);

\draw [dotted] \loopleftside;
\draw [dotted] \looprightside;

\draw (-3,1) -- (0,0);
\draw (-3,1) -- (0,2);

\draw[thick,->] (-1.5,3.5) -- (-1.5, 2.5);

\fill (-3,1) circle (0.030cm);

\begin{scope}[shift={(0,4)}]

\def\triangleFD{(-3,1) -- (0,0) -- (-0.2,0.75) -- cycle}
\def\triangleBD{(-3,1) -- (0,0) -- (0.2,1.25) -- cycle}
\def\triangleFU{(-3,1) -- (-0.2,0.75) -- (0,2) -- cycle}
\def\triangleBU{(-3,1) -- (0.2,1.25) -- (0,2) -- cycle}

\path [pattern=north east lines, pattern color=black!10!white] \triangleBD;
\path [pattern=north west lines, pattern color=black!10!white] \triangleBU;
\path [fill=white] \triangleFU;
\path [fill=white] \triangleFD;
\path [pattern=north west lines, pattern color=black!30!white] \triangleFD;
\path [pattern=north east lines, pattern color=black!30!white] \triangleFU;

\draw (-3,1) -- (0,0);
\draw (-3,1) -- (0,2);
\draw (-3,1) -- (-0.2,0.75);
\draw [dashed, color=black!20] (-3,1) -- (0.2,1.25);
\draw [dotted] (0,0) -- (-0.2,0.75);
\draw [dotted] (-0.2,0.75) -- (0,2);
\draw [dotted] (0,0) -- (0.2,1.25);
\draw [dotted] (0.2,1.25) -- (0,2);
\fill (-3,1) circle (0.030cm);

\end{scope}

\end{tikzpicture}

\end{center}

\caption{Two barycentric subdivisions of the Artin fan of a logarithmic curve with a single node, illustrated using generalized cone complexes  (Example \ref{Ex:nodal}).}\label{Fig:nodal-blowup}

\end{minipage}\begin{minipage}[b]{.5\textwidth}

\begin{center}

\begin{tikzpicture}

\begin{scope}[shift={(0,4)}]

\draw [white, pattern=north west lines, pattern color=black!20!white] (0,0) -- (2,0) -- (2,2) -- (0,0);

\fill (0,0) circle (0.050cm);

\draw  (0,0) -- (2,2);
\draw (0,0) -- (2,0);

\end{scope}

\draw[thick,->] (1,3.5) -- (1, 2.5);

\draw [white, pattern=north west lines, pattern color=black!20!white] (0,0) -- (2,0) -- (2,2) -- (0,0);

\fill (0,0) circle (0.050cm);

\draw [dashed] (0,0) -- (2,2);
\draw (0,0) -- (2,0);

\end{tikzpicture}

\end{center}

\caption{Barycentric subdivision of $\cA^{[2]}$, the Artin fan of the  Whitney umbrella, illustrated using generalized cone complexes   (Example \ref{Ex:whitney}).}\label{Fig:whitney-blowup}

\end{minipage}

\end{figure}

\begin{theorem} \label{thm:subdiv}
Any quasi-compact Artin fan $\cX$ has a projective subdivision $\cY \to \cX$ admitting a strict map  $\cY \to \cA^{n}$, for some integer $n$. 
\end{theorem}
\begin{proof}

{\sc Step 1.} Let $\cX$ be an Artin fan.  Consider  the barycentric subdivision, $B\cX \to \cX$, a projective morphism by Proposition \ref{prop:subdiv-projective}.  By Lemma~\ref{lem:barycentric-stable}, every vector of $B\cX$ is stable. 

{\sc Step 2.} Since every vector in $B\cX$ is stable we can resolve singularities by the same procedure used to resolve singularities in toric and toroidal geometry, see e.g. \cite[Theorem 11*, page 94]{KKMS}.  Indeed, if $\cA_\sigma \rightarrow \cX$ is any strict map, then any element of $\sigma$ corresponds to a stable vector of $\cX$, hence can be used as the center of a star subdivision.  We can apply the familiar procedure to resolve toric singularities individually to each map $\cA_\sigma \rightarrow \cX$.  With further star subdivision we maintain the property that every vector is stable.  After a finite number of subdivisions, we obtain a subdivision $\cY \rightarrow \cX$ where $\cY$ is smooth (and in particular simplicial) with the property that every vector $y : \cA \rightarrow \cY$ is stable. Since the procedure involves only star subdivisions, it is projective.

{\sc Step 3.}
For each ray $x : \cA \rightarrow \cY$, we construct a map $\cY \rightarrow \cA$.  If $\cA_\sigma \rightarrow \cY$ (where $\sigma$ is necessarily isomorphic to $\NN^r$, so $\cA_\sigma \simeq \cA^r$) is a 
\change{strict, \'etale map}
through which $x$ factors, then $x$ factors as the inclusion of a ray of $\sigma$.  This factorization $\cA \rightarrow \cA_\sigma \rightarrow \cY$ is unique since all vectors of $\cY$ are stable (as a consequence of Steps 1 and 2).  We have a canonical projection $\sigma = \NN^r \rightarrow \NN$ onto this ray inducing a map $\cA_\sigma \rightarrow \cA$.  This projection is compatible with restriction to an open subset of $\cA^r$ through which $x$ also factors.  It restricts to the projection to the generic point of $\cA$ on any open set of $\cA^r$ not containing $x$.

If $\cA_\sigma \rightarrow \cY$ is a chart through which $x$ does not factor, then we take $\cA_\sigma \rightarrow \cA$ to be the projection to the generic point.  As the choice of lift of $x$ is unique when it exists, these definitions are unambiguous and hence glue to give a map $\cY \rightarrow \cA$.

Repeating this construction for every ray of $\cY$ we get a map $\cY \rightarrow \cA^n$ where $n$ is the number of rays of $\cY$.  We verify that it is strict:  On a chart $\cA^r \rightarrow \cY$, the map $\cA^r \subset \cA^n$ is, by construction, the identity on the $r$ factors in the domain, and the projection to the generic point at the remaining factors.  In particular, it is an open embedding and therefore strict.

\end{proof}

\begin{remark} When the Artin fan $\cX$ has a cover by \emph{open} Artin \emph{sub}cones, the argument of \cite[Lemma 8.7]{AMR} can be used to show that the resulting strict morphism $\cY \to \cA^n$ is an open embedding.  If $\cX$ has faithful monodromy, then \change{the proof of} Theorem~\ref{thm:subdiv} shows that the barycentric subdivision of $\cX$ has an open cover by subcones.  Combining this with the previous observation, we observe that if $\cX$ is an Artin fan without monodromy then applying Theorem~\ref{thm:subdiv} after barycentric subdivision yields an open substack of $\cA^n$. This is illustrated in Figures \ref{Fig:nodal-blowup} and \ref{Fig:whitney-blowup}. One barycentric subdivision suffices for our theorem, but in Figure \ref{Fig:nodal-blowup} a second barycentric subdivision is needed to embed the complex as a fan.
\end{remark}

With notation as in the statement and proof of the theorem, consider the set $S' \subset S$ consisting of strict maps $\cA \rightarrow \cY$ such that the composite $\cA\to \cX$ is not strict. These correspond to the exceptional divisors of $\cY \to \cX$: each of the generators of $\cA^{S'}$ pulls back to a line bundle and section $(\cL_i, \sigma_i)$ on $\cY$ vanishing along an exceptional divisor of the projective morphism $\cY \rightarrow \cX$.  

\begin{corollary}\label{cor:target-modify}
Let $X$ be a noetherian logarithmic stack whose logarithmic strata are locally connected in the smooth topology.  Then there is a logarithmic modification $\Psi : Y \rightarrow X$ with relatively ample line bundle $L$, as well as line bundles and sections $(L_i, s_i)$ on $Y$ vanishing along substacks $E_i \subset Y$, having the following properties:
\begin{enumerate}[label=(\roman{*})]
 \item \label{cor:target-modify:1} The morphism $\Psi$ is projective, logarithmically  \'etale, and surjective.
 \item \label{cor:target-modify:2} $\Psi$ is an isomorphism away from the  locus $\cup_{i}E_{i}$.
 \item \label{cor:target-modify:3} We have $L = \bigotimes L_i^{\otimes m_i}$ with $m_i$ negative.
 \item \label{cor:target-modify:4} $Y$ has Deligne-Faltings logarithmic structure.
 \item \label{cor:target-modify:5} If $X$ is logarithmically  smooth, then the underlying structure $\uY$ is smooth in the usual sense.
 
\end{enumerate}
\end{corollary}
\begin{proof}
Let $\cX$ be the Artin fan of $X$, let $\cY$ be given by Theorem \ref{thm:subdiv}  and take $Y = \cY \fpr_{\cX} X$.  By the theorem, this gives~\ref{cor:target-modify:1},~\ref{cor:target-modify:4}, and~\ref{cor:target-modify:5} immediately.  For the $L_i$, $s_i$, and $E_i$ we simply pull back $\cL_i$, $\sigma_i$, and $\cE_i$ from $\cY$.  This gives~\ref{cor:target-modify:2}.   Recall that the exceptional divisor of any star subdivision is anti-ample. Since the composition of projective morphisms is projective, there is a linear combination, with positive coefficients, of the pullbacks of these divisors which is anti-ample for $\cY \to \cX$. Since every divisor $(s_i)$ corresponding to an element of $S'$ appears in such an exceptional divisor,   there exist  negative integers $m_i$ such  that $\cL = \bigotimes \cL_i^{\otimes m_i}$ is relatively ample for $\cY \to \cX$. Then ~\ref{cor:target-modify:3} is obtained by taking $L$ to be the pull-back of $\cL$.
\end{proof}

\subsection{Stable maps into subdivisions}

Let $\cX$ be an Artin fan.  Recall from Section~\ref{sec:introprestablemaps} that $\fM(\cX)$ is the moduli stack parameterizing prestable logarithmic maps to $\cX$ and $\fM'(\cY \to \cX)$ parameterizes prestable maps which are relatively stable for $\cY \to \cX$.  An object of $\fM'(\cY \to \cX)(S)$ is a diagram
\begin{equation}\label{relmap}
\vcenter{\xymatrix{
C\ar[r]\ar[d]&\cY\ar[d]\\
\oC\ar[r]&\cX
}}
\end{equation}
of prestable logarithmic maps over $S$ where $C\to\oC$ is a logarithmic modification and the automorphism group of this diagram relative to the bottom arrow $\oC\to\cX$ is finite.  In other words, the map $C \rightarrow \cY \fpr_{\cX} \oC$ is stable and $C \rightarrow \oC$ is a contraction of rational components.  The morphism $\fM'(\cY \to \cX) \to \fM(\cX)$ under consideration sends a diagram~\eqref{relmap} to $\oC\to\cX$.  See \cite[\change{Sections~3 and~4}]{AWnew}, where this morphism is shown to be birational, for a more thorough discussion.  

\begin{proposition} \label{prop:alteration-lift}
Let $\cY \rightarrow \cX$ be a modification  of Artin fans.  Any diagram
\begin{equation*} \xymatrix{
& & \fM'(\cY \rightarrow \cX) \ar[d] \\
S' \ar@{-->}@/^15pt/[urr]^-{\exists!} \ar@{-->}[r]_{\exists} & S \ar[r]  &\fM(\cX)
} \end{equation*}
admits a unique lift after passing to a (not necessarily representable) logarithmic modification $S' \rightarrow S$.
\end{proposition}
\begin{proof}
The map $S \rightarrow \fM(\cX)$ corresponds to a logarithmic curve $\oC$ over $S$ and a map $\oC \rightarrow \cX$.  Applying Corollary~\ref{cor:Artin-fan-functoriality} to the map $\oC \rightarrow S \times \cX$ we obtain a diagram of Artin fans~\eqref{eqn:7},
\begin{equation} \label{eqn:7} \vcenter{\xymatrix{
& \cY \ar[d] \\
\ocC \ar[r] \ar[d] & \cX \\
\cS
}} \end{equation}
\change{where $\cS$ is the Artin fan of $S$ and $\oC$ is the relative Artin fan of $\ocC$ over $\cS \times \cX$.}
Take $\cC = \ocC \fpr_{\cX} \cY$, with the fiber product formed in the category of fine, saturated, logarithmic algebraic stacks; this is the pullback of a subdivision of $\cX$, hence is a subdivision of $\ocC$, \change{and in particular has connected fibers over $\ocC$}.  After a logarithmic modification of $\cS$, we can assume that $\cS$ is smooth (Theorem~\ref{thm:subdiv}), $\cC \rightarrow \cS$ is equidimensional \cite[Lemmas~4.1 and~4.3]{AK}, and therefore that $\cC$ is flat over $\cS$ \cite[Remark~4.6]{AK}.  By \cite[Proposition~5.1]{AK} we can ensure as well that the fibers of $\cC \rightarrow \cS$ are reduced by replacing the integral lattice of $\cS$ with a finite index sublattice.%
\footnote{Note that this corresponds to a root stack construction, so that $S' \rightarrow S$ is not necessarily representable.}

Now let $C = \cC \fpr_{\ocC} \oC$.  We show that $C$ is a logarithmic curve \cite[Definition~4.5]{ACGHOSS} over $S$.  We must verify the following properties:
\begin{enumerate}
\item $C$ is logarithmically smooth over $S$:  It is the composition of a logarithmically \'etale map $C \rightarrow \oC$ (the base change of the log\change{arithmically} \'etale map $\cY \rightarrow \cX$) and a logarithmically smooth map $\oC \rightarrow S$.
\item $C \rightarrow S$ has connected fibers:  Since $\oC$ has connected fibers over $S$, it is sufficient to show that $C \rightarrow \oC$ has connected fibers.  This follows by \change{strict} base change from the connectedness of the fibers of 
\change{$\cC \rightarrow \ocC$.}
\item $C \rightarrow S$ is integral in the logarithmic sense:  Since $C \rightarrow \cC$ and $S \rightarrow \cS$ are strict, this is immediate from the flatness of the map $\cC \rightarrow \cS$.
\item $C \rightarrow S$ has reduced, $1$-dimensional fibers:  The map $C \rightarrow \cC \fpr_{\cS} S$ is smooth of relative dimension~$1$ and $\cC \rightarrow \cS$ has reduced $0$-dimensional fibers.
\item $C$ is proper over $S$:  The map $\cC \rightarrow \ocC$ is proper (it is a subdivision), so $C \rightarrow \oC$ is proper, and $\oC \rightarrow S$ is proper by hypothesis.
\end{enumerate}
Therefore $C$ lifts $\oC \rightarrow \cX$ to a diagram~\eqref{relmap}.  It is the base change of a subdivision, so it is a logarithmic modification.  Furthermore, any component of $C$ contracted in $\oC$ is stabilized by the map to $\cY$.  Therefore this diagram lifts $\oC \rightarrow \cX$ to a point of $\fM'(\cY \rightarrow \cX)$.

We verify that this lift is unique.  Suppose that $C'$ is another lift.  By the universal property of fiber product, we obtain a map $f : C' \rightarrow C = \oC \fpr_{\cX} \cY$.  By the definition of $\fM'(\cY \rightarrow \cX)$, this map is stable.  On the other hand, the map $C' \rightarrow \oC$ is a logarithmic modification of logarithmic curves, hence is a contraction of semistable components.  Thus, $C' \to C$ is stable and a contraction of semistable components, hence is an isomorphism.
\end{proof}

\begin{corollary}\label{Cor:valuative}
Assume that $\cY$ is a subdivision of an Artin fan $\cX$.  Then the morphism
\begin{equation*}
\fM'(\cY \rightarrow \cX) \rightarrow \fM(\cX)
\end{equation*}
is birational and satisfies the valuative criterion for properness.
\end{corollary}
\begin{proof}
Birationality was proved in \cite[\change{Proposition~5.2.1}]{AWnew}.  The valuative criterion is immediate from Proposition~\ref{prop:alteration-lift}.
\end{proof}

\subsection{The valuative criterion}\label{Sec:separation}
\begin{proof}[Proof of Proposition \ref{Prop:separation}]
 Let $R$ be a discrete valuation ring and $K$ be the fraction field of $R$. Consider   an object $\uf: \spec K \to \wchoice{\mathnormal}{\smash{\ocM(X)}}{\wunderbar}$, which we would like to extend, possibly after base change, to a unique  object $\spec R \to \wchoice{\mathnormal}{\smash{\ocM(X)}}{\wunderbar}$. The object $\uf$ corresponds to a \change{logarithmic structure $M_S$ on $\uS = \Spec K$, and a} logarithmic morphism $f: S  \to \ocM(X)$, \change{where $S := (\Spec K,M_S)$}.
 
\change{Choose a projective subdivision $\cY \rightarrow \cX$ as in Theorem~\ref{thm:subdiv}, and let $Y = X \mathop\times_{\cX} \cY$.}
 Consider the  the composition $\mathfrak f:S\to \fM(\cX)$ of $f$ with $\ocM(X)\to \fM(\cX)$, and the cartesian diagram \ref{eqn:3}. 
By  Proposition \ref{prop:alteration-lift} there is a logarithmic modification $S' \to S$ and a unique lift of $\mathfrak f$ to  $\mathfrak f': S'\to \fM'(\cY \rightarrow \cX)$, giving rise to a unique lift $f':  S'\to \ocM(Y)$.  
\change{As $S'$ is of finite type, it has a $K$-point, at least after replacing $K$ by a finite extension.  We let $S''$ be this point, with the logarithmic structure restricted from $S'$.}  By the valuative criterion of $\ocM(Y)$ \cite[Corollary 3.11]{AC} after 
\change{replacing $R$ and $K$ with a finite extension}
we have a logarithmic scheme $\change{T''} = (\Spec R, M_{\change{T''}})$ extending $S''$, and a unique extension $\tilde f':\change{T''} \to  \ocM(Y)$ of $f':S' \to  \ocM(Y)$. Composing with $\ocM(Y) \to \ocM(X)$ we obtain an arrow $\tilde f: \change{T''} \to  \ocM(X)$:
 
\begin{equation}\label{diag:valuative}
\xymatrix{
S'' \ar[rrd]|-(.35){f} \ar[rr]^{f'} \ar[d] && \ocM(Y) \ar[d] \\
\change{T''}  \ar@{-->}[rr]_{\tilde f} \ar@{-->}[rru]|-(.35){\tilde f'} && \ocM(X)  
}
\end{equation}

\change{It still remains to show that two extensions $\tilde f_1, \tilde f_2 : \uT = \Spec R \rightarrow \underline{\smash{\ocM(X)}}$ extending $f$ must agree.  It is sufficient to verify this after a finite base change.  We give $\uT$ the
logarithmic structure $M_T$ pulled back from the map $(\tilde f_1, \tilde f_2) : \uT \rightarrow \ocM(X) \times \ocM(X)$.  According to Proposition~\ref{prop:alteration-lift}, there is a logarithmic modification $T' \rightarrow T$ after which there is a unique
lift of $T' \rightarrow \ocM(Y)$ of the composition $T' \rightarrow T \rightarrow \ocM(X)$.  Now, $\uT' \rightarrow \uT$ is surjective and proper, so after a finite base change, it admits a section.  Therefore $\tilde f_1$ and $\tilde f_2$ both lift to $\underline{\smash{\ocM(Y)}}$, hence coincide because $\underline{\smash{\ocM(Y)}}$ is proper.}

\end{proof}

\section{Boundedness of numerical data}

In this section we will identify locally constant numerical data $\Gamma$ on $\ocM(X)$ such that each $\ocM_\Gamma(X)$ is of finite type.  In addition to the genus $g$ of the source curve, the number $n$ of marked points, and the homology class $\beta$ of the curve's image in $X$, we also have evaluation maps
\begin{equation*}
\ocM(X) \rightarrow \wedge X \rightarrow \wedge \cX
\end{equation*}
associated to each marking, see Section \ref{sec:intrologpts}.  The choice of a connected component of $\wedge \cX$ for each marked point gives one more locally constant datum.  Let $\Gamma = (g, n, \beta, \varphi)$ where $\varphi \in \pi_0(\wedge \cX)^{n}$.  We write $\ocM_\Gamma(X)$ for the open and closed substack of $\ocM(X)$ with these numerical data.

Select a logarithmic modification $Y \rightarrow X$, obtained by base change from a subdivision of Artin fans $\cY \rightarrow \cX$, as in Corollary~\ref{cor:target-modify}.  The irreducible components $\cE_i$ of the exceptional locus of $\cY \to \cX$ are nonsingular divisors which are unions of logarithmic strata.  We denote their pre-images on $Y$ by $E_i$ and the corresponding line bundles by $L_i$. 

Write $\ocM_\Gamma(Y)$ for the open and closed substack of $\ocM(Y)$ lying above $\ocM_\Gamma(X)$.  The following proposition, whose proof occupies the rest of this section, will complete the proof of our main theorem:

\begin{proposition}\label{Prop:boundedness}
The algebraic stack $\ocM_\Gamma(Y)$ is of finite type.
\end{proposition}

Recall that if the genus $g$, number of markings $n$, degree with respect to some ample line bundle on $Y$, and a component of \change{$\wedge Y$} for each marked point are fixed in $\Xi$, then $\ocM_\Xi(Y)$ is of finite type \cite[Theorem~3.12]{GS} or \cite[Corollary 3.13]{AC}. We will show that $\ocM_\Gamma(Y)$ is a union of only finitely many $\ocM_\Xi(Y)$.  Obviously, once $\Gamma$ is fixed, $g$ and $n$ are fixed.  The first step of our argument will be to show that the components of $\wedge \cY$ map bijectively to the components of $\wedge \cX$, so that once a component of $\wedge \cX$ is fixed in $\Gamma$ there is a unique component of $\wedge \cY$ lying above it.  Finally, we will show that the degree in $Y$ is bounded by the choice of $\Gamma$.

\subsection{Boundedness of contact orders}
\label{sec:contact-orders}

Recall that a family of logarithmic points parameterized by a logarithmic scheme $(X,M_X)$ is simply a line bundle $L$ on $X$.  Equivalently, a family of logarithmic points parameterized by $X$ may be viewed as an augmentation $M_X \rightarrow M'_X$ of the logarithmic structure of $X$ with $\oM'_X = \oM_X \times \NN$.  A logarithmic point of $(Y,M_Y)$ parameterized by $(X,M_X)$ is a logarithmic morphism $(X,M'_X) \rightarrow (Y,M_Y)$, where $(X,M'_X)$ is a family of logarithmic points parameterized by $(X,M_X)$.

\begin{proposition} \label{prop:log-points-cartesian}
\begin{enumerate} \item Suppose that the following diagram of logarithmic algebraic stacks is cartesian:
\begin{equation*} \xymatrix{
Y' \ar[r] \ar[d] & Y \ar[d] \\
X' \ar[r] & X
} \end{equation*}
Then the diagram below is cartesian as well:
\begin{equation*} \xymatrix{
\wedge Y' \ar[r] \ar[d] & \wedge Y \ar[d] \\
\wedge X' \ar[r] & \wedge X
} \end{equation*}
\item 
\change{Suppose $X' \to X$ is strict. Then the diagram 
$$\xymatrix{\wedge X' \ar[r]\ar[d] & \wedge X \ar[d] \\
 \uX' \ar[r] & \uX}$$ is cartesian as well.}
\end{enumerate}
\end{proposition}
\begin{proof}
\change{The first statement is immediate from the modular description of the stack of logarithmic points. The second follows from the first.}
\end{proof}

We evaluate $\wedge \cA_\sigma$ for a fine, saturated, sharp monoid $\sigma$.  \change{Our calculation is analogous to \cite[Section~3.2]{Gillam}, which treats affine toric varieties.}  Consider a map $f : (X,M'_X) \rightarrow \cA_\sigma$.  This corresponds to a homomorphism of monoids,
\begin{equation} \label{eqn:8}
\sigma^\vee \rightarrow \Gamma(X, \oM_X) \times \Hom(X, \NN) .
\end{equation}
The map $\sigma^\vee \rightarrow \Hom(X,\NN)$ may be viewed as a locally constant function $\varphi : X \rightarrow \sigma$.  For each $\varphi \in \sigma$ we therefore obtain an open and closed substack $\wedge_{\varphi} \cA_\sigma$. The element $\varphi$ is called the {\em contact order}.

\change{Consider the closed substack $\bgm \subset \cA$. The stack $\uBGm$ is the stack of log points over log schemes with the universal family $\bgm \to \uBGm$ \cite[Section~2.3]{ACGM}. By the definition of $\cA_{\sigma}$, the morphism $\mathrm{id}\times\varphi: \sigma^{\vee} \to \sigma^{\vee}\times \NN$ defines a morphism of logarithmic stacks $h: \cA_{\sigma}\times \bgm \to \cA_{\sigma}$, which is a logarithmic point in $\cA_{\sigma}$ over $\cA_{\sigma}\times\uBGm$.  This defines a tautological morphism
\begin{equation} \label{equ:logpt}
\cA_\sigma \times \uBGm \rightarrow \wedge_\varphi \cA_\sigma
\end{equation}
}

\change{
Conversely, consider any logarithmic point $f: (X,M'_X) \to \cA_{\sigma}$ with contact order $\varphi : X \rightarrow \sigma$. We obtain a morphism $X \to \uBGm$ induced by the family of  logarithmic points over $X$. On the other hand, $f$ induces the composition $\sigma^{\vee} \to \oM'_X \simeq \NN \times \oM_X \to \oM_X$, hence a morphism $X \to \cA_{\sigma}$. This gives a morphism $\tilde{h}: X \to \cA_{\sigma}\times\uBGm$ such that the logarithmic point $f$ is the pull-back of $h$ via $\tilde{h}$. In particular, the tautological morphism $\cA_{\sigma}\times\uBGm \to \wedge_{\varphi}\cA_{\sigma}$ is surjective. This defines another morphism
\begin{equation}\label{equ:logpt-reverse}
\wedge_{\varphi}\cA_{\sigma} \to \cA_{\sigma}\times\uBGm.
\end{equation}
One checks that (\ref{equ:logpt}) and (\ref{equ:logpt-reverse}) are inverse of each other. We have just proved the following proposition:
}
\change{
\begin{proposition} \label{prop:components}
For any $\varphi\in\sigma$, the stack $\wedge_{\varphi}\cA_{\sigma}$ is isomorphic to $\cA_\sigma \times \uBGm$.  In particular, it is irreducible and of finite type.
\end{proposition}
}

\begin{corollary}
The connected components of $\wedge \cA_\sigma$ are in bijection with the elements of $\sigma$.
\end{corollary}

\begin{corollary}
If $\cX$ is an Artin fan then the connected components of $\wedge \cX$ are in bijection with the isomorphism classes of maps $\cA \rightarrow \cX$.
\end{corollary}
\begin{proof}
We may present $\cX$ as a colimit of a diagram of strict maps among the $\cA_\sigma$.  Since $\pi_0(-)$, $\wedge(-)$, and $\Hom(\cA,-)$ all respect strict colimits of Artin fans, the problem reduces to the case $\cX = \cA_\sigma$.  In that case we only need to recall that $\Hom(\cA, \cA_\sigma) = \sigma$, functorially in $\sigma$.
\end{proof}

\begin{corollary}
If $\sigma \rightarrow \tau$ is a morphism of fine, saturated, sharp monoids and $\varphi \in \sigma$ has image $\psi \in \tau$ then the induced map $\wedge_\varphi \cA_\sigma \rightarrow \wedge_\psi \cA_\tau$ is of finite type.
\end{corollary}
\begin{proof}
\change{
The statement follows from the identifications $\wedge_\varphi \cA_\sigma \simeq \cA_\sigma \times \uBGm$ and $\wedge_\psi \cA_\tau \simeq \cA_\tau \times \uBGm$ and the fact that $\cA_\sigma \rightarrow \cA_\tau$ is of finite type.}
\end{proof}

\begin{corollary} \label{cor:log-point-comps}
Let $\cY \rightarrow \cX$ be a subdivision.  Then the induced map $\wedge \cY \rightarrow \wedge \cX$ 
is of finite type.
\end{corollary}
\begin{proof}
A subdivision induces a bijection
\begin{equation*}
\Hom(\cA, \cY) \rightarrow \Hom(\cA, \cX) 
\end{equation*}
on the sets of connected components \change{of $\wedge \cY $ and $\wedge \cX$}.  To show the map is of finite type, it is sufficient to work \'etale-locally in $\cX$.  We may therefore assume $\cX = \cA_\tau$.  The subdivision $\cY$ of $\cA_\tau$ has an open cover by finitely many $\cA_\sigma$.  This reduces us to showing that the maps $\wedge_\varphi \cA_\sigma \rightarrow \wedge_\varphi \cA_\tau$ are of finite type, as we did in the previous corollary.
\end{proof} 
\change{
\begin{corollary} \label{cor:log-point-comp-ft} Suppose that $\cX$ is a quasicompact Artin fan.  Then each connected component of $\wedge \cX$ is of finite type.
\end{corollary}
\begin{proof}  
Suppose that $\varphi \in \Hom(\cA, \cX)$ corresponds to a connected component $\wedge_\varphi \cX$ of $\wedge \cX$.  Then $\wedge_\varphi \cX$ is covered by the maps $\wedge_{\psi} \cA_\sigma \rightarrow \wedge_\varphi \cX$ as $\psi$ ranges over lifts $\cA \xrightarrow{\psi} \cA_\sigma$ of $\varphi$ along strict, \'etale maps $\cA_\sigma \rightarrow \cX$.  But $\cX$ has a cover by finitely many strict, \'etale maps $\cA_\sigma \rightarrow \cX$, so $\wedge_\varphi \cX$ has a cover by finitely many strict, \'etale maps $\wedge_\psi \cA_\sigma$, hence is of finite type.  
\end{proof}
}

\change{
\begin{proof}[Proof of Proposition~\ref{Prop:component-bounded}]
As $\uX \rightarrow \underline{\cX}$ is of finite type and, by Corollary~\ref{cor:log-point-comp-ft}, each connected component of $\wedge \cX$ is of finite type, it follows from Proposition~\ref{prop:log-points-cartesian} that each connected component of $\wedge X$ is of finite type.  Likewise, $\wedge \cY \rightarrow \wedge \cX$ is of finite type by Corollary~\ref{cor:log-point-comps}, so $\wedge Y \rightarrow \wedge X$ is of finite type, again by Proposition~\ref{prop:log-points-cartesian}.
\end{proof}
}

\subsection{Boundedness of the curve classes}

Let $f : C \rightarrow Y$ be an object of $\ocM_\Gamma(Y)$.  Denote by $c_{j}(E_{i})$ the contact order of the $j$-th marking with the exceptional divisor $E_{i}$ as in Corollary~\ref{cor:target-modify}.  These numbers are uniquely determined by the induced maps to $\wedge \cY$, hence by $\Gamma$.

The following is a restatement of Proposition \ref{Prop:deg-constant}.  Recall that $L$ is a relatively ample line bundle for $Y$ over $X$ and that $L \simeq \bigotimes L_i^{\otimes m_i}$ for negative integers $m_i$.

\begin{proposition}\label{lem:curve-boundedness}
Let $f : C \rightarrow Y$ be a point of $\ocM(Y)$.  The values $c_j(E_i)$ determine $\deg_C(L)$.
\end{proposition}
\begin{proof}
We have 
\begin{equation*}
\deg_C(L) = \sum_{i}m_i\deg_C(E_{i}) = \sum_{i}m_i\deg_C(\cE_{i}).
\end{equation*}
This quantity is locally constant on $\fM(\cY)$, which is logarithmically smooth \cite[\change{Proposition~1.6.1}]{AWnew}.  We can therefore replace $C$ with a deformation that is smooth and intersects the $\cE_i$ properly.  In this case $\deg_{C}(\cE_{i}) = \sum_j c_{j}(E_{i})$, so  $\deg_C(L) = \sum_{i,j}m_ic_{j}(E_{i})$ is determined by the $c_{j}(E_{i})$, as required.
\end{proof}

\begin{proposition}\label{Prop:bound-class}
Fix an ample line bundle $M$ on $X$ and $f : C \rightarrow Y$ a point of $\ocM(Y)$.  Denote the projection from $Y$ to $X$ by $\pi$.  Then $L \otimes \pi^\ast M$ is ample on $Y$ and
\begin{equation*}
\deg(f^\ast (L \otimes \pi^\ast M)) = \deg(f^\ast \pi^\ast M) + \sum_{i,j} m_i c_j(E_i) .
\end{equation*}
In particular, the degree of $f$ with respect ot $L \otimes \pi^\ast M$ is determined combinatorially by the image of $(C,f)$ in $\ocM(X)$.
\end{proposition}
\begin{proof}
We have
\begin{equation*}
\deg(f^\ast L \otimes f^\ast \pi^\ast M) = \deg(f^\ast \pi^\ast M) + \deg(f^\ast L)
\end{equation*}
and $\deg(f^\ast L)$ was computed in the last proposition.
\end{proof}

We conclude that $\Gamma$ bounds the degree of $f : C \rightarrow Y$ as well as its contact orders along the logarithmic divisors.  Therefore $\ocM_\Gamma(Y)$ is of finite type.  This completes the proof of Proposition \ref{Prop:boundedness}.

As the map
\begin{equation*}
\ocM_\Gamma(Y) \rightarrow \ocM_\Gamma(X)
\end{equation*}
is proper, we deduce that $\ocM_\Gamma(X)$ is of finite type as well.

\bibliographystyle{amsalpha}             

\bibliography{logbd}       

\end{document}